\documentclass[12pt]{amsart}

\usepackage[draft]{graphicx}

\usepackage{amssymb}
\usepackage{amsmath}
\usepackage{amstext}
\usepackage{amsbsy}
\usepackage{amsxtra}
\usepackage{amsfonts}
\usepackage{latexsym}
\usepackage{alltt}
\usepackage{bbm}
\usepackage{url}

\usepackage{verbatim}
\usepackage{amscd}
\usepackage[all]{xy}

\theoremstyle{plain}
\newtheorem{theorem}{Theorem}[section]
\newtheorem{corollary}[theorem]{Corollary}
\newtheorem{lemma}[theorem]{Lemma}
\newtheorem{proposition}[theorem]{Proposition}

\newtheorem{example}[theorem]{Example}

\newcommand{\mi}{\mathbbm i}
\newcommand{\C}{\mathbb C}
\newcommand{\R}{\mathbb R}

\newcommand{\tr}{\mathrm{tr}}

\newcommand{\su}{\mathfrak{su}_{\mbox{\tiny{2}}}}

\numberwithin{equation}{section}

\DeclareMathOperator{\RE}{Re}
\DeclareMathOperator{\IM}{Im}

\title[CMC surfaces with integrable boundary conditions]{On constant mean curvature surfaces satisfying integrable boundary conditions}


\begin{document}

\author{Martin Kilian}
\address{Martin Kilian, University College Cork, Ireland.}


\begin{abstract} We consider the local theory of constant mean curvature surfaces that satisfy one or two integrable boundary conditions and determine the corresponding potentials for the generalized Weierstrass representation. 
\end{abstract}

\maketitle
\section*{Introduction}
Imposing additional properties, like symmetries, embeddedness, simply-connectedness or compactness on a constant mean curvature surface often severely limits the classes of examples.  Historically, the first reduction by Delaunay \cite{Del} resulted in classifying the rotationally symmetric examples. A few decades later, Enneper \cite{Enn, Enn1} and Dobriner \cite{Dob, Dob1} considered surfaces foliated by planar or spherical curvature lines. Modern accounts and further developments can be found in the works of Abresch \cite{Abr} and Wente \cite{Wen, Wen:enn} amongst many others, where the structure equations are reduced to elliptic ordinary differential equations. 

Here we consider the local theory of constant mean curvature surfaces that satisfy one or two integrable boundary conditions. The integrable boundary condition reflects the geometric property that the surface meets a sphere at a constant angle along the curve of intersection. From our integrable systems point of view this means studying the corresponding 'potentials' in the generalized Weierstrass representation of Dorfmeister, Pedit and Wu \cite{DorPW}. 

By a result of \cite{KilS2021} all such surfaces are of finite type, so their potentials are polynomial, and the purpose here is to determine all such potentials. We first study the additional so called 'K-symmetry' on the space of potentials that comes from the integrable boundary condition. In Theorem \ref{th:imaginary} we show that K-symmetry halves the dimension of the space of potentials of arbitrary fixed degree. We show in Corollary \ref{th:off-diagonalKsymmetry} that all off-diagonal K-symmetric potentials of arbitrary high degree yield degree one surfaces, that is, associated family members of Delaunay surfaces. In the second part we consider potentials that give solutions of the elliptic $\sinh$-Gordon equation satisfying two integrable boundary conditions, that is free boundary constant mean curvature surfaces with two boundary components. We do not consider period problems here, but focus on the algebraic conditions for the local theory of such surfaces.

\section{Integrable boundary conditions}
Let $U \subset \R^2$ be an open, connected and simply-connected set containing the origin. 
An associated family $\mathbf{f}_\lambda:U \to \su \cong \R^3,\,\lambda \in \mathbb{S}^1$ of conformaly immersed surfaces with constant mean curvature $H \neq 0$ can be framed by an $\mathrm{SU}_2$--valued \emph{extended frame} $\mathbf{F}_\lambda$ . It satisfies a linear partial differential system $\mathbf{F}_\lambda^{-1} d\mathbf{F}_\lambda = \mathbf{\Omega}_\lambda$. Here $\mathbf{\Omega}_\lambda = \mathbf{U}_\lambda dx + \mathbf{V}_\lambda dy$ is a $2 \times 2$ matrix-valued 1-form which encodes the surface invariants. For the induced metric $\exp [2\omega] \,(dx^2 + dy^2)$ the $\sinh$-Gordon equation
\[
	\Delta \,\omega + \sinh \omega = 0
\]
is the integrability condition 
\[
	2d\mathbf{\Omega}_\lambda + [\mathbf{\Omega}_\lambda \wedge \mathbf{\Omega}_\lambda ] = 0
\]
for $\mathbf{\Omega}_\lambda$. There is a loop group method to obtain integrable connections via a generalized Weierstrass representation \cite{DorPW}. For us it suffices, due to a result in \cite{KilS2021}, to restrict to the finite gap setting, for which we invoke the Symes method \cite{BurP_adl}:  extended frames can be obtained from polynomial Killing fields, and these in turn are determined by their intial conditions, the finite gap potentials. A polynomial Killing field is a solution of a Lax equation
\[
	d\eta = [\eta,\,\Omega (\eta)] \qquad \eta(0) = \xi_\lambda\,.
\]
The initial condition (or potential) $\xi_\lambda$ is a $\mathrm{sl}_2 (\C)$--valued Laurent-polynomial $\lambda \mapsto \xi_\lambda$ and $\eta \mapsto \Omega (\eta)$ is a projection, and $\Omega = \Omega(\eta)$ integrates to an extended frame $\mathbf{F}_\lambda^{-1} d\mathbf{F}_\lambda = \Omega$ of a constant mean curvature surface. The relationship between extended frame and polynomial Killing field is $\eta_\lambda = \mathbf{F}_\lambda^{-1} \xi_\lambda \mathbf{F}_\lambda$. Writing $\partial_\lambda$ for differentiation with respect to $\lambda$,  the corresponding associated family is given by 
\begin{equation} \label{eq:SymBobenko}
	\mathbf{f}_\lambda = - 2\,\mi \lambda H^{-1} (\partial_\lambda \mathbf{F}_\lambda) \,\mathbf{F}_\lambda^{-1}\,.
\end{equation}
After a choice of normalizations as in \cite{HKS1} for the surface invariants, $\mathbf{U}_\lambda$ takes the form
\[
	\mathbf{U}_\lambda= \frac{\mathbbm{i}}{8} \left(
\begin{array}{cc} -2 \omega_y  &  e^\omega \lambda^{-1} + e^{-\omega }  \\
 e^{-\omega } + e^\omega \lambda  & 2  \omega_y \\
\end{array}
\right)\,.
\]
Note that 
\begin{equation*}
 \overline{\mathbf{U}_{\bar\lambda}}^t  = - \mathbf{U}_{\lambda^{-1}}  
	\qquad \mbox{ and} \qquad \overline{\mathbf{F}_{1/\bar\lambda}}^t = \mathbf{F}_\lambda ^{-1}\,.
\end{equation*}
We study solutions $\omega: U \to \mathbb{R}$ of the sinh-Gordon equation that satisfy an additional \emph{integrable boundary condition}
\begin{equation} \label{eq:boundary}
	\omega_y  = e^\omega A + e^{-\omega} B \qquad \mbox{ along \quad $y=0$}\,.
\end{equation}
Here $A$ and $B$ are arbitrary real numbers. Also, when we write 'along $y=0$' we mean its restriction to the domain, so $\{ y=0 \} \cap U$.
\subsection{K-matrix} Define for $A,\,B \in \mathbb{R}$ and $\lambda \in \mathbb{C}^\ast$ the K-matrix
\begin{equation} \label{eq:K}
	K(\lambda)  = \begin{pmatrix}
 4 A  -4 B \lambda  & \lambda - \lambda^{-1} \\
 \lambda -  \lambda^{-1} & 4 A  - 4 B \lambda^{-1}
\end{pmatrix}
\end{equation}
Note that $K(\pm 1) = 4(A \mp B)\,\mathbbm{1}$, as well as
\begin{equation} \label{eq:Kproperties}
	K(\lambda) = K(\lambda)^{\,t} \quad \mbox{ and } \quad \overline{K(\bar\lambda)} = K(\lambda) \quad \mbox{ and } \quad K(\lambda^{-1}) = \mathrm{adj}\,K(\lambda)\,. 
\end{equation}
\begin{lemma} \label{lem:Ucomm1}
For $\lambda \neq \pm 1$ we have $\omega_y = e^\omega A + e^{-\omega} B$ along $y=0$ if and only if 
\begin{equation} \label{eq:symU}
	K\,\mathbf{U}_\lambda =  \mathbf{U}_{\lambda^{-1}} \,K  \,.
\end{equation}
\end{lemma}
\begin{proof}
Compute $K \,\mathbf{U}_\lambda - \mathbf{U}_{\lambda^{-1}} \,K  = \tfrac{\mathbbm{i}}{2}  (\lambda^{-1} -\lambda ) \left( \omega_y -  \left(e^\omega A + e^{-\omega} B \right)\right) \bigl( 
\begin{smallmatrix}
 0 & -1\\
 1 & 0 \\
\end{smallmatrix}
\bigr)$.
\end{proof}
\begin{lemma} \label{lem:Fswitch2}
Suppose $\tfrac{d}{dx} \mathbf{F}_\lambda = \mathbf{F}_\lambda \mathbf{U}_\lambda,\,\mathbf{F}_\lambda(0) = \mathbbm{1}$. Then along $y=0$ the condition \eqref{eq:symU} holds if and only if 
\begin{equation} \label{eq:KF}
	 K\,\mathbf{F}_\lambda =  \mathbf{F}_{\lambda^{-1}}\,K \qquad \mbox{ along \quad $y=0$}\,.
\end{equation}
\end{lemma}
\begin{proof}
Suppose along $y=0$ the condition \eqref{eq:symU} holds. Consider $\mathbf{G}_\lambda = K\,\mathbf{F}_\lambda\,K^{-1}$. Then $\tfrac{d}{dx} \mathbf{G}_\lambda = \mathbf{G}_\lambda \mathbf{U}_{\lambda^{-1}},\,\mathbf{G}_\lambda (0) = \mathbbm{1}$. By uniqueness of solutions $\mathbf{G}_\lambda = \mathbf{F}_{\lambda^{-1}}$, and thus $K\,\mathbf{F}_\lambda\,K^{-1} = \mathbf{F}_{\lambda^{-1}}$.

Conversely, suppose \eqref{eq:KF} holds. Differentiating with respect to $x$ gives $K\,\mathbf{F}_\lambda' =  \mathbf{F}_{\lambda^{-1}}'\,K$ and hence $K\,\mathbf{F}_\lambda \mathbf{U}_\lambda = \mathbf{F}_{\lambda^{-1}}\mathbf{U}_{\lambda^{-1}} \,K$. Using  \eqref{eq:KF} again gives the claim.
\end{proof}	
\begin{corollary}
Let $\mathbf{F}_\lambda$ be an extended frame satisfying \eqref{eq:KF} along $y=0$, and $\zeta_\lambda = \mathbf{F}_\lambda^{-1} \xi_\lambda \mathbf{F}_\lambda$ a polynomial Killing field. Then along 
$y=0$ we have 
\begin{equation} \label{eq:PfKsym1}
	K \zeta_\lambda + \overline{\zeta_{\bar\lambda}}^t K = \mathbf{F}_{\lambda^{-1}}^{-1} \bigl( \, K\,\xi_\lambda + \overline{\xi_{\bar{\lambda}}}^t K \,\bigr) \,\mathbf{F}_\lambda\,.
\end{equation}
\end{corollary}
A polynomial Killing field $\zeta_\lambda$ is K-symmetric (along $y=0$) if 
\begin{equation}\label{eq:PfKsym2}
	K \zeta_\lambda + \overline{\zeta_{\bar\lambda}}^t K = 0\,.
\end{equation}
By \eqref{eq:PfKsym1} this holds if and only if its potential $\xi_\lambda$ is K-symmetric
\begin{equation} \label{eq:Sklyanin-potentials}
	K\,\xi_\lambda + \overline{\xi_{\bar{\lambda}}}^t K = 0\,.
\end{equation}
\subsection{K-symmetry and the Symes map} The Symes map $\xi_\lambda \mapsto \mathbf{F}_\lambda$  utilizes the Iwasawa factorization \cite{PreS} at each $z \in U$, and we write
\begin{equation} \label{eq:Iwasawa}
	\mathbf{\Phi}_\lambda (z) = \exp [z\,\xi_\lambda ] = \mathbf{F}_\lambda (z,\bar z) \,\mathbf{B}_\lambda (z,\bar z)\,.
\end{equation}
With $\mathbf{F}_\lambda^{-1} d\,\mathbf{F}_\lambda = \mathbf{\Omega}_\lambda$ we have the gauge relation
\begin{equation} \label{eq:gauge}
	\xi_\lambda = \mathbf{B}_\lambda^{-1} \mathbf{\Omega}_\lambda \mathbf{B}_\lambda + \mathbf{B}_\lambda^{-1} d\,\mathbf{B}_\lambda\,.
\end{equation}
Next, we determine how \rm{K}-symmetry \eqref{eq:Sklyanin-potentials} on a potential descends to the maps $\mathbf{\Phi}_\lambda$ and $\mathbf{B}_\lambda$. Lemma  \ref{lem:Fswitch2} already dealt with the extended frame $\mathbf{F}_\lambda$, as its \rm{K}-symmetry \eqref{eq:KF} came directly from its differential equation.
\begin{lemma} \label{th:KPhi}
Let $\xi_\lambda$ be a K-symmetric potential. Then  
\begin{equation} \label{eq:PhiK-symmetry}
	K\,\mathbf{\Phi}_\lambda (z) = \overline{\mathbf{\Phi}_{\bar\lambda} (\bar z)} ^{t^{-1}} K \qquad \mbox{for all $z \in U$}\,.
\end{equation}
\end{lemma}
\begin{proof}
We compute $\overline{\mathbf{\Phi}_{\bar\lambda}(\bar z) }^{\,t} = \exp[z \,\overline{\xi_{\bar\lambda}}^{\,t}] = \exp[ - z\,K \,\xi_\lambda K^{-1} ] = K \,\mathbf{\Phi}_{\lambda}^{-1}(z)\, K^{-1}$.
\end{proof}
\begin{lemma}
Let $\xi_\lambda$ be a {\rm{K}}-symmetric potential, and $ \exp [z\,\xi_\lambda ] = \mathbf{F}_\lambda  \,\mathbf{B}_\lambda$ the pointwise Iwasawa factorization for $z \in U$. Assume that 
$K\,\mathbf{F}_\lambda =  \mathbf{F}_{\lambda^{-1}}\,K$ along $y=0$. Then 
\begin{equation}\label{eq:KB}
	K\,\mathbf{B}_\lambda = \overline{\mathbf{B}_{\bar\lambda}} ^{t^{-1}} K \qquad \mbox{ along \quad $y=0$}\,.
\end{equation}
\end{lemma}
\begin{proof}
By Lemma \ref{th:KPhi} we have $K\,\mathbf{\Phi}_\lambda = \overline{\mathbf{\Phi}_{\bar\lambda}} ^{t^{-1}} K$ and thus  
\begin{equation*} 
	K\,\mathbf{F}_\lambda \,\mathbf{B}_\lambda = \overline{\mathbf{F}_{\bar\lambda}} ^{t^{-1}}  \overline{\mathbf{B}_{\bar\lambda}} ^{t^{-1}} K \,.
\end{equation*}
Using the assumption on $\mathbf{F}_\lambda$ along $y=0$, and that $\overline{\mathbf{F}_{\bar{\lambda}}}^t = \mathbf{F}_{\lambda^{-1}}^{-1}$ we obtain \eqref{eq:KB}.
\end{proof}
%
%
\subsection{K-symmetry and the associated family} Finally, we look at K-symmetry of the associated family. Let $\partial_\lambda$ denote differentiation with respect to $\lambda$. 
\begin{lemma} \label{th:KSymmetrySB}
Assume that $K\,\mathbf{F}_\lambda =  \mathbf{F}_{\lambda^{-1}}\,K$ along $y=0$. Then for the associated family \eqref{eq:SymBobenko} we have along $y=0$ that
\begin{equation} \label{eq:Ksym-f}
	K\,\mathbf{f}_\lambda + \overline{\mathbf{f}_\lambda}^t K = -2\mi\lambda H^{-1} \left( K( \partial_\lambda 	\mathbf{F}_\lambda ) -  (\partial_\lambda \mathbf{F}_{\lambda^{-1}}) K \right) \mathbf{F}_\lambda^{-1} \,.
\end{equation}
\end{lemma}
\begin{proof}
Differentiating $K\,\mathbf{F}_\lambda =  \mathbf{F}_{\lambda^{-1}}\,K$ with respect to $\lambda$ gives
\begin{equation} \label{eq:DiffKF}
	 \mathbf{F}_{\lambda^{-1}}\,\partial_\lambda K - \partial_\lambda K\,\mathbf{F}_\lambda = K\partial_\lambda \mathbf{F}_\lambda  -  \partial_\lambda \mathbf{F}_{\lambda^{-1}}\,K\,.
\end{equation}
Inserting $\mathbf{F}_\lambda = K^{-1} \mathbf{F}_{\lambda^{-1}} K$ into \eqref{eq:SymBobenko} yields
\[
	\mathbf{f}_\lambda = K^{-1} \mathbf{f}_{\lambda^{-1}} K  - 2\mi\lambda H^{-1} \left( K^{-1} \mathbf{F}_{\lambda^{-1}} \partial_\lambda K \mathbf{F}_\lambda^{-1}  - K^{-1} \partial_\lambda K \right)
\]
and thus 
\[
	K \mathbf{f}_\lambda  - \mathbf{f}_{\lambda^{-1}} K  	= - 2\mi\lambda H^{-1} \left( \mathbf{F}_{\lambda^{-1}} \partial_\lambda K   -  \partial_\lambda K \mathbf{F}_\lambda \right) \mathbf{F}_\lambda^{-1}\,.
\]
Using \eqref{eq:DiffKF} and $\mathbf{f}_{\lambda^{-1}} = -  \overline{\mathbf{f}_\lambda}^t$ proves the claim.
\end{proof}
We say that an associated family is K-symmetric along $y=0$ if 
\begin{equation*}
	K\,\mathbf{f}_\lambda + \overline{\mathbf{f}_\lambda}^t K = 0\,.
\end{equation*}
By equation \eqref{eq:Ksym-f} this holds if and only if $K( \partial_\lambda 	\mathbf{F}_\lambda ) = (\partial_\lambda \mathbf{F}_{\lambda^{-1}}) K$. 
%
%
\section{K-symmetric potentials}
Expanding the matrix of a finite-gap potential as 
\[
	\xi_\lambda = \sum_{k=-1}^d \hat{\xi}_k \lambda^k\,,
\]
then its coefficient matrices $\hat{\xi}_k \in \mathrm{sl}_2 (\C)$ satisfy the \emph{reality condition}
\begin{equation} \label{eq:potential-reality}
	\hat{\xi}_k = - \overline{\hat{\xi\,}}_{d-k-1}^t \quad -1 \leq k \leq d\,.
\end{equation}
To match up with the structure of $\mathbf{U}_\lambda$ we have
\begin{equation} \label{eq:residue}
	\hat{\xi}_{-1} \in \begin{pmatrix} 0 & \mi\R_+ \\ 0 & 0 \end{pmatrix}\,,
\end{equation}
which ensures that also $\hat{\xi}_d \neq 0$. We say that $\xi_\lambda$ has \emph{degree} $d$. Let us write
\begin{equation*}
	\xi_\lambda = \begin{pmatrix} \alpha_\lambda & \beta_\lambda \\ \gamma_\lambda & -\alpha_\lambda \end{pmatrix}
\end{equation*}
with entries
\begin{equation*} 
	 \alpha_\lambda = \sum_{k=0}^{d-1} \alpha_k \lambda^k \,,\qquad \beta_\lambda = \sum_{k=-1}^{d-1} \beta_k \lambda^k \,\,,\qquad \gamma_\lambda = \sum_{k=0}^d \gamma_k \lambda^k \qquad \mathrm{ for } \quad \alpha_k,\, \beta_k,\,\gamma_k \in \C\,.
\end{equation*}
The reality condition \eqref{eq:potential-reality} on the coefficients read
\begin{equation}\label{eq:reality}
	\gamma_k =-\overline{\beta}_{d-k-1} \qquad \mathrm{and} 	\qquad \alpha_k = -\overline{\alpha}_{d-k-1}
\end{equation}
The latter splits into real and imaginary parts
\begin{equation*} 
	\RE [ \alpha_k + \alpha_{d-k-1}]  = 0 = 
	\IM [ \alpha_k - \alpha_{d-k-1}]  \,.
\end{equation*}
When degree $d$ is odd, then for $k=(d-1)/2$ also $d-k-1= (d-1)/2$ and hence
\begin{equation*}
	\RE \alpha_{(d-1)/2} = 0 \,.
\end{equation*}
Note that there is no freedom in the choice of $\gamma_\lambda$, as all its coefficients are determined by the coefficients of $\beta_\lambda$. There are $2d+1$ real degrees of freedom for $\beta_\lambda$, and $d$ real degrees of freedom for $\alpha_\lambda$ when $d$ is even, and $d-1$ when $d$ id odd. Hence the space of potentials of degree $d$ is an open subset of a $3d+1$ dimensional real vectorspace when $d$ is even, and an open subset of a $3d$ dimensional real vectorspace when $d$ is odd.
 

In terms of the four matrix entries [ij] in the order [11], [12], [21] respectively [22], equation \eqref{eq:Sklyanin-potentials} reads
\begin{equation}\label{eq:Sklyanin-potential-symmetry}
\begin{split}
(\lambda - \lambda^{-1})\left( \gamma_\lambda + \overline{\gamma_{\bar{\lambda}}} \right) +4( A - B\lambda ) \left(   \alpha _\lambda + \overline{\alpha_{\bar{\lambda}}}  \right)  &\equiv 0 \,,
\\
-(\lambda - \lambda^{-1})\left( \alpha_\lambda - \overline{\alpha_{\bar{\lambda}}} \right) + 4( A -  B\lambda )\beta_\lambda + 4(A-B\lambda^{-1})  \overline{\gamma_{\bar{\lambda}}}   &\equiv 0  \,,
\\
(\lambda - \lambda^{-1})\left( \alpha_\lambda - \overline{\alpha_{\bar{\lambda}}} \right) + 4( A -  B\lambda )\overline{\beta_{\bar{\lambda}}} + 4(A-B\lambda^{-1})  \gamma_\lambda   &\equiv 0  \,,
\\
(\lambda - \lambda^{-1})\left( \beta_\lambda + \overline{\beta_{\bar{\lambda}}} \right) - 4( A -  B\lambda^{-1} ) \left(   \alpha _\lambda + \overline{\alpha_{\bar{\lambda}}}  \right)  &\equiv 0  \,.
\end{split}
\end{equation}
Inspection of the two diagonal conditions in \eqref{eq:Sklyanin-potential-symmetry} proves
\begin{proposition} \label{th:1st}
For the entries of a \rm{K}-symmetric potential we have 
\begin{enumerate}
\item $\displaystyle{\alpha_\lambda + \overline{\alpha_{\bar{\lambda}}}  \equiv 0  \Longleftrightarrow
\beta_\lambda + \overline{\beta_{\bar{\lambda}}} \equiv 0} \Longleftrightarrow \gamma_\lambda + \overline{\gamma_{\bar{\lambda}}}  \equiv 0$.
\item $\displaystyle{\alpha_\lambda \equiv 0  \Longrightarrow
\beta_\lambda + \overline{\beta_{\bar{\lambda}}} \equiv 0}$.
\end{enumerate}
\end{proposition}
In each of the four equations \eqref{eq:Sklyanin-potential-symmetry} we have powers of $\lambda$ ranging from $\lambda^{-1}$ to $\lambda^{d+1}$. Comparing coefficients gives for $k=-1,\,\ldots,\,d+1$ the system
\begin{equation}\label{eq:Sklyanin-polynomial}
\begin{split}
-\left( \overline{\gamma}_{k+1} + \gamma_{k+1} \right) + 4 A \left(\overline{\alpha}_k + \alpha _k \right) +  \overline{\gamma}_{k-1} + \gamma_{k-1} - 4 B \left( \overline{\alpha}_{k-1} + \alpha_{k-1} \right)   &= 0 
\\
\alpha_{k+1} - \overline{\alpha}_{k+1} - 4 B \overline{\gamma}_{k+1}   + 4 A (\overline{\gamma}_k  +  \beta_k ) +  \overline{\alpha}_{k-1}  - \alpha_{k-1} - 4 B \beta_{k-1}  &= 0
\\
\overline{\alpha}_{k+1} - \alpha_{k+1} - 4 B \gamma_{k+1}  + 
4 A (\overline{\beta}_k  +  \gamma_k )  +  \alpha_{k-1}  - \overline{\alpha}_{k-1} - 4 B \overline{\beta}_{k-1}  &= 0
\\
4 B ( \overline{\alpha}_{k+1} + \alpha_{k+1} ) - \overline{\beta}_{k+1}  - \beta_{k+1}  - 4 A (\overline{\alpha}_k  + \alpha_k )  +    \beta_{k-1}  + \overline{\beta}_{k-1}  &= 0\,.
\end{split}
\end{equation}
Splitting into real part $\RE$ and imaginary part $\IM$ gives
\begin{equation} \label{eq:split_equations} \begin{split}
-  \RE\gamma_{k+1} + 4 A \RE \alpha _k  +  \RE\gamma_{k-1} - 4 B \RE \alpha_{k-1}  &= 0  \\
\mi\IM \alpha_{k+1} - 2 B \overline{\gamma}_{k+1} + 2 A (\overline{\gamma}_k  +  \beta_k ) - \mi\IM\alpha_{k-1} - 2 B \beta_{k-1} &= 0 \\
-\mi\IM \alpha_{k+1} - 2 B \gamma_{k+1}  + 
2 A (\overline{\beta}_k  +  \gamma_k )  + \mi\IM \alpha_{k-1} - 2 B \overline{\beta}_{k-1}  &= 0 \\
4 B \RE\alpha_{k+1} - \RE\beta_{k+1}  - 4 A \RE \alpha_k   +  \RE \beta_{k-1}  &= 0 \,.\end{split}
\end{equation}
%
%
\subsection{Real Parts.} Adding the first and last equations in \eqref{eq:split_equations} gives the recursion 
\begin{equation} \label{eq:alpha_recursion}
	4B\RE [\alpha_{k+1} - \alpha_{k-1} ]   -\RE [\beta_{d-k} - \beta_{d-k-2} ] + \RE [ \beta_{k-1} - \beta_{k+1} ] = 0\,.
\end{equation}
Plugging in $k=-1$ gives
\[
	4B \RE \alpha_0 = \RE \beta_0
\]
Plugging in $k=0$ gives
\[
	4B \RE \alpha_1 + \RE \beta_{d-2} + \RE [\beta_{-1} - \beta_1 ] = 0
\]
Hence $\RE\alpha_0,\,\RE\alpha_1$ are uniquely determined by $\RE\beta_0,\,\RE\beta_1$ and $\RE\beta_{d-2}$. Now the recursion \eqref{eq:alpha_recursion} uniquely determines the remaining $\RE\alpha_k$.  We summarize in the following
\begin{proposition} 
The real parts of the coefficients of $\beta_\lambda$ uniquely determine the real parts of the coefficients of $\alpha_\lambda$. If the coefficients of $\beta_\lambda$ are all imaginary, then also all the coefficients of $\alpha_\lambda$ are imaginary.
\end{proposition}
%
%
\subsection{Imaginary Parts.} 
Adding and subtracting the two middle equations in \eqref{eq:split_equations} and using \eqref{eq:reality} that $\gamma_k =-\overline{\beta}_{d-k-1}$ we get recursions
\begin{equation} \label{eq:beta_recursion} \begin{split}
	 A\,\RE [\beta_k - \beta_{d-k-1} ]  + B\,\RE [ \beta_{d-k-2}  - \beta_{k-1} ] &= 0\,,\\
	\IM[\alpha_{k+1} - \alpha_{k-1}]+2A\IM[\beta_k-\beta_{d-k-1}] + 2 B \IM [\beta_{d-k-2}-\beta_{k-1}] &=0\,.
\end{split}
\end{equation}
and combining gives 
\begin{equation} \label{eq:beta_recursion1}
	\mi\IM[\alpha_{k+1} - \alpha_{k-1}]+2A(\beta_k-\beta_{d-k-1}) + 2 B (\beta_{d-k-2}-\beta_{k-1}) =0\,.
\end{equation}
For $k=-1$ we obtain
\begin{equation} \label{eq:1stbetarecursion} \begin{split}
	A\,\RE \beta_{-1}  + B\,\RE \beta_{d-1} &= 0\,, \\
	\mi\IM \alpha_0 + 2A \beta_{-1}  + 2B \beta_{d-1} &= 0\,, 
\end{split}
\end{equation}
and since $\RE\beta_{-1} = 0$ by \eqref{eq:residue} it follows that 
\begin{equation} \label{eq:re-top}
	\RE \beta_{d-1} = 0\,,
\end{equation}
and the two imaginary numbers $\beta_{-1},\,\beta_{d-1}$ determine the imaginary part of $\alpha_0$.

For $k=0$ we obtain
\begin{equation*} 
\begin{split}
	A\RE \beta_0 + B \RE \beta_{d-2} &= 0\,, \\
	\mi \IM \alpha_1 + 2A (\beta_0 - \beta_{d-1})  + 2B(\beta_{d-2} - \beta_{-1}) &= 0\,, 
\end{split}
\end{equation*}
Hence $\RE\beta_0$ determines $\RE\beta_{d-2}$, and the imaginary number $A\beta_0  + B \beta_{d-2}$ determines $\IM\alpha_1$. 

Next for $k=1$ we get 
\begin{equation*}
\begin{split}
	A\RE[\beta_1 - \beta_{d-2}] + B \RE [\beta_{d-3} - \beta_0] &= 0\,,\\
	\mi \IM [\alpha_2 - \alpha_0] + 2A (\beta_1 - \beta_{d-2})  + 2B(\beta_{d-3} - \beta_0) &= 0\,, 
\end{split}
\end{equation*}
and therefore $\RE\beta_1$ determines $\RE\beta_{d-3}$. Likewise, $\RE\beta_2$ determines $\RE\beta_{d-4}$. Thus $\RE\beta_k$ determines $\RE\beta_{d-k-2}$. When the degree is even $d=2N$, then the last degree of freedom is given by $\RE\beta_{N-2}$ which determines $\RE\beta_N$. 
Hence choices for $\RE\beta_0, \RE\beta_1,\,\ldots ,\, \RE\beta_{N-2}$ determine the remaining  $\RE\beta_{N-1},\,\ldots ,\, \RE_{d-2}$. Thus for even degree there are $(d-2)/2$ choices. Likewise, when the degree is odd, there are $(d-1)/2$ choices. 

When $d$ is odd, then for $k=(d+1)/2$ we have $d-k-1 = (d-3)/2$ and hence
\begin{equation*}\begin{split}
	\RE  [ \alpha_{(d+1)/2} + \alpha_{(d-3)/2}] &= 0\,,\\
	\IM [ \alpha_{(d+1)/2} - \alpha_{(d-3)/2}]  &= 0\,.
\end{split}
\end{equation*}
Now $(d+1)/2 -  (d-3)/2 = 2$, so these two indices are two apart, but the recursion \eqref{eq:beta_recursion} gives no additional condition, as it reads
\begin{equation*} 
	 \IM [\alpha_{(d+1)/2} - \alpha_{(d-3)/2} ] + 2A (\beta_{(d-1)/2} - \beta_{(d-1)/2}) + 2B (\beta_{(d-3)/2} - \beta_{(d-3)/2})  = 0 \,.
\end{equation*}
Therefore there is no additional constraint on the imaginary parts of the coefficients of $\beta_\lambda$ for the case $d$ odd. \\
If $d$ is even then the indices $k$ and $d-k-1$ are never two apart, so the recursion \eqref{eq:beta_recursion} in conjunction with the reality condition gives no additional constraint in this case. There is a further one-dimensional reduction on the coefficients of $\beta_\lambda$ when the degree is even due to the following
\begin{proposition} 
Let $\xi_\lambda$ be a \rm{K}-symmetric potential with non-zero diagonal. If the degree $d$ is even, then 
\begin{equation*}
	\sum_{k=-1}^{d-1} (-1)^{k+1}  \beta_k = 0\,.
\end{equation*}
\end{proposition}
\begin{proof}
When $d$ is even, the recursion \eqref{eq:beta_recursion1} gives 
\begin{equation*} \begin{split}
	\mi\IM\alpha_0 &= -2A\,\beta_{-1} - 2B\,\beta_{d-1} \\
	\mi\IM\alpha_1 &= - 2 A (\beta_0  -\beta_{d-1}) - 2 B(-\beta_{-1} + \beta_{d-2}) \\
		\vdots & \\
	\mi\IM\alpha_{d-2} &= -2A(\beta_{-1} + \beta_1 - \beta_2 +  \ldots -\beta_{d-2} ) - 2B( - \beta_0 + \beta_1 -  \ldots+ \beta_{d-3} + \beta_{d-1})  \\
	\mi\IM\alpha_{d-1} &= -2A(\beta_0 - \beta_1 + \beta_2 - \ldots - \beta_{d-1}) - 2B (-\beta_{-1} + \beta_0 - \beta_1 + \ldots + \beta_{d-2})\,.
\end{split}
\end{equation*}
The symmetry $\IM\alpha_k = \IM\alpha_{d-k-1}$ for all $k=0,\,\ldots,\,d-1$ gives the claim. 
\end{proof}
In conclusion, when $d$ is even, there are $d$ choices, and when $d$ is odd there are $d+1$ real choices for $\IM\beta_{-1},\,\ldots,\,\IM \beta_{d-1}$ that uniquely determine $\IM\alpha_0,\,\ldots,\,\IM\alpha_{d-1}$. This proves
\begin{proposition}\label{th:beta_count}
Let $\xi_\lambda$ be a \rm{K}-symmetric potential of degree $d$. 

\rm{(i)} If $d$ is even, there are $(d-2)/2$ real degrees of freedom for the real parts of the coefficients of $\beta_\lambda$. If $d$ is odd, there are $(d-1)/2$ real degrees of freedom for the real parts of the coefficients of $\beta_\lambda$. 

\rm{(ii)} The imaginary parts of the coefficients of $\beta_\lambda$ uniquely determine the imaginary parts of the coefficients of $\alpha_\lambda$, and there are $d+1$ many real choices when $d$ is odd, and $d$ many when $d$ is even.
\end{proposition}
%
%
Adding up parameter counts from Proposition \ref{th:beta_count}  shows that K-symmetry halves the dimension of the parameter space. More precisely, we have proven 
\begin{theorem} \label{th:imaginary}
The space of \rm{K}-symmetric potentials of degree $d$ is real $(3d-2)/2$--dimensional for even $d$, and $(3d+1)/2$--dimensional for odd $d$.
\end{theorem}
\begin{example} \label{ex:im} (i) For $d=1$ we have $\RE \beta_0 = 0,\,\RE\alpha_0 = 0$ and $\IM \alpha_0 = -2 (A\beta_{-1} + B\beta_0)$. Hence degree one potentials all have purely imaginary entries and have the form
\begin{equation*}
	\xi_\lambda = \mi \begin{pmatrix} -2(A\,\beta_{-1} + B\,\beta_0) & \beta_{-1} \lambda^{-1} + \beta_0 \\ \beta_0 + \beta_{-1} \lambda & 2(A\,\beta_{-1} + B\,\beta_0)  \end{pmatrix}
\end{equation*}
with $\beta_{-1}, \,\beta_0 \in \R$ free, $\beta_{-1} > 0$. 

(ii) For $d=2$ we have $\RE\beta_1 = 0$ and $(A+B)\RE\beta_0 = 0$ so either $\RE\beta_0 = 0$ or $B=-A$. Further, $\RE\alpha_0 = -\RE\alpha_1 = \RE\beta_0/(4B)$ and $\mi\IM\alpha_0 = -2(A\beta_{-1} +B\beta_1)$ as well as $\mi\IM\alpha_1 = -2(A(\beta_0 - \beta_1) + B(\beta_0 - \beta_{-1}))$. Thus reality $\overline{\alpha_0} = -\alpha_1$ holds if and only if $$\beta_{-1} - \beta_0 + \beta_1 = 0\,,$$
and again all entries are purely imaginary and given by 
\begin{equation*} \begin{split}
	\alpha_\lambda &= -2(A\beta_{-1} + B\beta_1) (1+\lambda) \,, \\
	\beta_\lambda &= \beta_{-1}\lambda^{-1} +(\beta_{-1} + \beta_1) + \beta_1\lambda \,.
\end{split}
\end{equation*}
The general form of degree two \rm{K}-symmetric potentials is
\begin{equation*}
	\xi_\lambda = \mi \begin{pmatrix} -2 (A\beta_{-1} + B \beta_1)(1+\lambda) & \beta_{-1}( \lambda^{-1} + 1) + \beta_1 (1+ \lambda) \\ \beta_1(1+\lambda) + \beta_{-1}(1+ \lambda^2) &  
2(A \beta_{-1} + B\beta_1)(1+\lambda) \end{pmatrix}\,.
\end{equation*}
(iii) For $d=3$ we have $\RE\beta_2 = 0$ and $A\RE\beta_0 + B\RE\beta_1 = 0$. Hence $\RE\beta_1 = -(A/B) \RE\beta_0$. Further, $\RE\alpha_0 = -\RE\alpha_2 = (1/4B)\RE\beta_0$ and $\RE\alpha_1 = 0$. Hence $\RE\beta_0 \neq 0$ guarantees that not all coefficients are purely imaginary. Continuing, we get $\mi\IM\alpha_0 = \mi\IM\alpha_2 = -2(A\beta_{-1} + B\beta_2),\,\mi\IM\alpha_1 = 2(A\beta_2 + B\beta_{-1})$. Thus in the degree 3 case we see that it is possible to get K-symmetric potentials which do not have purely imaginary entries, and
\begin{equation*} \begin{split}
	\alpha_0 &= 1/(4B)\RE\beta_0  -2(A\beta_{-1} + B\beta_2) \,,\\
	\alpha_1 &= 2(A\beta_2 + B\beta_{-1}) \,,\\
	\alpha_2 &=   -1/(4B)\RE\beta_0 - 2(A\beta_{-1} + B \beta_2 )\,.
\end{split}
\end{equation*}
Here $\beta_{-1}, \beta_2 \in \mi\R$, so also $\alpha_1 \in \mi\R$. But $\beta_0,\,\beta_1 \in \C$ in general, subject only to $A\RE\beta_0 + B\RE\beta_1 = 0$, and thus also $\alpha_0,\,\alpha_2 \in \C$ in general. 
\end{example}
%
%
\section{Imaginary potentials}
By Proposition \ref{th:1st} \rm{(i)}, a K-symmetric potential satisfying $\alpha_\lambda  + \overline{\alpha_{\bar{\lambda}}} \equiv 0$ has purely imaginary coefficients so that
\begin{equation}\label{eq:Sklyanin-reality}
\alpha_\lambda  + \overline{\alpha_{\bar{\lambda}}} \equiv 
\beta_\lambda  + \overline{\beta_{\bar{\lambda}}} \equiv
\gamma_\lambda + \overline{\gamma_{\bar{\lambda}}} \equiv 0\,.
\end{equation}
Hence we can write such an \emph{imaginary potential} as 
\[
\xi_\lambda  = \mi \begin{pmatrix} \alpha_\lambda & \beta_\lambda \\ \gamma_\lambda & -\alpha_\lambda \end{pmatrix}  
\]
where $\alpha_\lambda,\,\beta_\lambda,\,\gamma_\lambda$ all have real coefficients. Now the recursion  \eqref{eq:beta_recursion1} reads
\begin{equation} \label{eq:beta_recursion2}
	\alpha_{k+1} - \alpha_{k-1} + 2A(\beta_k-\beta_{d-k-1}) + 2 B (\beta_{d-k-2}-\beta_{k-1}) =0\,.
\end{equation}
For $k=-1$ we obtain 
\begin{equation}\label{eq:alpha0} 
	\alpha_0 = -2A\,\beta_{-1} - 2B\,\beta_{d-1}\,.
\end{equation}
\begin{example} Since \rm{K}-symmetric potentials of degree one and two are imaginary and computed in example \ref{ex:im}, the first notable difference occurs for $d=3$: \\
{\rm{(i)}} For $d=3$ the recursion gives
\begin{equation*} \begin{split}
	\alpha_0 &= -2A\beta_{-1} - 2 B \beta_2 \\	
	\alpha_1 &= - 2A(\beta_0 - \beta_2) - 2B(-\beta_{-1}  + \beta_1)  \\
	\alpha_2 &= -2A\beta_{-1} - 2 B \beta_2 \,.
\end{split}
\end{equation*}
Hence reality $\alpha_0 = \alpha_2$ is satisfied without a further condition. \\
{\rm{(ii)}} For $d=4$ the recursion gives 
\begin{equation*} \begin{split}
	\alpha_0 &= -2A\beta_{-1} - 2 B \beta_3 \\	
	 \alpha_1 &= - 2A ( \beta_0  - \beta_3) -2B (- \beta_{-1}  + \beta_2) \\
	\alpha_2 &= -2 A (\beta _{-1}+\beta _1 - \beta _2) -2 B (-\beta_0 + \beta_1 + \beta_3)   \\
	\alpha_3 &= -2 A (\beta_0 - \beta_1 + \beta_2 - \beta_3 ) -2 B (-\beta_{-1} + \beta_0 - \beta_1 + \beta_2 ) 
\end{split}
\end{equation*}
The reality conditions $\alpha_0 = \alpha_3,\,\alpha_1 = \alpha_2$ hold for all $A,\,B$ if and only if
\begin{equation} \label{eq:beta_sum2}
	\beta_{-1} -\beta_0 + \beta_1 - \beta_2 + \beta_3 = 0\,.
\end{equation}
\end{example}	
%
%
%
\subsection{Off-diagonal K-symmetric potentials.}
By Proposition \ref{th:1st}  \rm{(ii)}, an off-diagonal K-symmetric potential is imaginary, so of the form
\begin{equation} \label{eq:offpot}
	\xi_\lambda  = \mi\, \begin{pmatrix} 0 & \beta_\lambda \\ \gamma_\lambda & 0 \end{pmatrix} \,.
\end{equation}
where $\beta_\lambda,\,\gamma_\lambda$ have only real coefficients. Now the recursion  \eqref{eq:beta_recursion1} reads
\begin{equation} \label{eq:beta_recursion3}
	 A(\beta_k-\beta_{d-k-1}) + B (\beta_{d-k-2}-\beta_{k-1}) =0 
\end{equation}
for $k=-1,\,\ldots,\,d+1$. Amongst the $d+1$ free paramters for off-diagonal, and hence imaginary $\beta_\lambda$, this recursion reduces the freedom by a further $d/2$ when $d$ is even, and $(d-1)/2$ for $d$ odd.  The first equations for $k=-1,\,\ldots,\,2$ read
\begin{equation*} \begin{split}
	A\beta_{-1} + B \beta_{d-1} &= 0 \\
	A(\beta_0 - \beta_{d-1}) + B( \beta_{d-2} - \beta_{-1} ) &= 0 \\
	A(\beta_1 - \beta_{d-2}) + B( \beta_{d-3} - \beta_0 ) &= 0 \\
	A(\beta_2 - \beta_{d-3}) + B( \beta_{d-4} - \beta_1 ) &= 0 
\end{split}
\end{equation*}
\begin{example}
We compute some low degree examples:\\
1. For $d=1$ we get  $\beta_0 = -\beta_{-1} A/B$, so degree one off-diagonal \rm{K}-symmetric potentials look like 
\begin{equation} \label{eq:deg1potential}
	\xi_\lambda =  \mi \beta_{-1} \begin{pmatrix} 0 & \lambda^{-1} -A/B \\ -A/B + \lambda & 0 \end{pmatrix}
\end{equation}
which drops rank at $\lambda = A/B,\,B/A$. \\
2. For $d=2$ we get $A\beta_{-1} + B \beta_1 = 0$ and $A(\beta_0 - \beta_1) + B( \beta_0 - \beta_{-1} ) = 0$, so degree two potentials are of the form
\begin{equation} \label{eq:deg2potential} \begin{split}
	\xi_\lambda &=  \mi \beta_{-1} \begin{pmatrix} 0 & \lambda^{-1} + (1-A/B)  - A/B \lambda  \\ -A/B + (1-A/B) \lambda +\lambda^2 & 0 \end{pmatrix} \\ 
		&=  \mi \beta_{-1} (\lambda + 1) \begin{pmatrix} 0 & \lambda^{-1}  - A/B \\\ -A/B +  \lambda & 0 \end{pmatrix} \end{split}
\end{equation}
Again $\det \xi_\lambda$ has simple roots at $\lambda = A/B,\,B/A$. In addition $\xi_\lambda \equiv 0$ at $\lambda = -1$. \\
3. For $d=3$ we can write $\beta_1,\,\beta_2$  in terms of $\beta_{-1},\,\beta_0$, in fact 
\[
	\beta_1 = (1-\tfrac{A^2}{B^2})  \beta_{-1} - \tfrac{A}{B} \beta_0\,\,,\qquad \beta_2 = - \tfrac{A}{B} \beta_{-1} \,.
\]
Inserting these expressions into $\beta_\lambda,\, \gamma_\lambda$, we observe that they have a common factor of degree 2, more specifically
\begin{equation*}
	\beta_\lambda / (\lambda^{-1} - A/B) = \gamma_\lambda /(\lambda - A/B) = \beta_{-1} \left(\lambda^2 + \tfrac{A}{B}\lambda + 1\right)+\lambda  \beta_0 \,.
\end{equation*}
Hence degree 3 off-diagonal \rm{K}-symmetric potentials are all of the form
\begin{equation} \label{eq:deg3potential} 
	\xi_\lambda =    \mi (\beta_{-1} \left(\lambda^2 + \tfrac{A}{B}\lambda + 1\right)+\lambda  \beta_0) \begin{pmatrix} 0 & \lambda^{-1}  - A/B \\\ -A/B +  \lambda & 0 \end{pmatrix} \,.
\end{equation}
This is the general behaviour of off-diagonal finite gap K-symmetric potentials, and the content of the next theorem.
\end{example}
\begin{theorem}
Let $\xi_\lambda$ be an off-diagonal \rm{K}-symmetric potential of degree $d$. Then there exists a unique real polynomial $\lambda \mapsto p(\lambda)$ of degree $d-1$ such that
\[
	\xi_\lambda = \mi\,p(\lambda) \begin{pmatrix} 0 & \lambda^{-1}  - A/B \\\ -A/B +  \lambda & 0 \end{pmatrix} \,.
\]
\end{theorem}
\begin{proof}
Write $\xi_\lambda$ as in \eqref{eq:offpot} with entries $\beta_\lambda,\,\gamma_\lambda$ with $\deg \gamma_\lambda = d$. Long division shows that the unique polynomial $\sum_{k=0}^{d-1} p_k \lambda^k$ of degree $d-1$ with coefficients recursivley defined by
\begin{equation*}
	p_0 = \beta_{-1}\,, \qquad
	p_k = \beta_{k-1} +\frac{A}{B}\, p_{k-1}  \quad \mathrm{for} \quad k=1,\,\ldots ,\,d-1
\end{equation*}
divides both $\beta_\lambda,\,\gamma_\lambda$ with quotients $\lambda^{-1} - A/B$ respectively $\lambda - A/B$.  Since all coefficients $p_k$ are real, this concludes the proof.
\end{proof}
The spectral curve $\Sigma$ is the 2-point compactification of 
\[
	\Sigma^* = \left\{ (\nu,\, \lambda ) \in \C^2 \mid \nu^2 = - \det \xi_\lambda \right\}\,.
\]
When $\xi = p \,\tilde{\xi}$ for a polynomial $p$, then after renormalization $\xi$ and $\tilde{\xi}$ give the same spectral curve, since the roots of $p^2$ all have even order. Any spectral curve comes from an isospectral family of potentials that all share the same determinant. In each isospectral family there are off-diagonal potentials, and by Proposition \ref{th:1st} \rm{(ii)} they are automatically imaginary.  
\begin{corollary} \label{th:off-diagonalKsymmetry}
All spectral curves of off-diagonal {\rm{K}}-symmetric potentials have genus one.  Higher spectral genus examples must come from potentials with non-zero diagonal. 
\end{corollary}
%
%
\section{Spectral Theory of K-matrices}
We compute the values of $\lambda$ for which a K-matrix fails to be invertible, and determine the corresponding kernels. We show that the eigenvectors of $K$ are $\lambda$-independent, and we compute the residues at the simple poles of $K^{-1}$. 
%
%
\subsection{Kernels.} Computing
\[
	-\lambda^2 \det K (\lambda) = \lambda ^4 + 16 A B \lambda ^3 - 2 \lambda ^2 \left(8 A^2+8 B^2+1\right)+16 A B \lambda +1
\]
and $\lambda \mapsto \det K(\lambda)$ has four simple roots 
\begin{equation} \begin{split} \label{eq:roots}
	\varrho &= \frac{2A - \sqrt{4 A^2+1}}{2 B + \sqrt{4 B^2+1}} \quad \mbox{ and } \quad 
	r  = \frac{2A + \sqrt{4 A^2+1}}{2 B + \sqrt{4 B^2+1}}\,, \\
	\varrho^{-1} &=\frac{2A + \sqrt{4 A^2+1}}{2 B-\sqrt{4 B^2+1}}   \quad \mbox{ and } \quad r^{-1} = \frac{2A - \sqrt{4 A^2+1}}{2 B-\sqrt{4 B^2+1}} \,.
\end{split}
\end{equation}
The null-set of $\det K$, that is 
\begin{equation}\label{eq:nullset}
	\mathcal{N} = \{ \varrho,\,r,\,\varrho^{-1},\,r^{-1} \}\,,
\end{equation}
is invariant under inversion as well as under the map $(A,\,B) \mapsto (-A,\,-B)$. 
The matrices $K(\lambda),\,\lambda \in \mathcal{N}$ drop rank, and kernels compute to  
\begin{equation} \label{eq:kernels} \begin{split}
	&\ker K (\varrho) = \ker K (r) =\ker \begin{pmatrix} 1 & 2 B + \sqrt{1 + 4 B^2} \\ 0 & 0 \end{pmatrix} \,,\\
	&\ker K (\varrho^{-1}) = \ker K (r^{-1}) =\ker \begin{pmatrix} 1 & 2 B - \sqrt{1 + 4 B^2} \\ 0 & 0 \end{pmatrix} \,.
\end{split}
\end{equation}
Note that 
\begin{equation*}
	\ker K (\varrho) \perp \ker K (\varrho^{-1}) 
\end{equation*}
and that both kernels are independent of the value of $A$ in the $K$-matrix.
%
%
\subsection{Spectrum.} $K$ has eigenvalues $\mu_\pm = \tfrac{1}{2} ( \tr K \pm \sqrt{\tr^2 K - 4 \det K})$ and compute to
\begin{equation} \label{eq:mu_pm} \begin{split}
	\mu_- (\lambda) &= 4 A  - 2 B \left(\lambda + \lambda^{-1} \right)  - \sqrt{4 B^2+1} \left(\lambda - \lambda^{-1} \right) \\
\mu_+ (\lambda) &= 4 A  - 2 B \left(\lambda + \lambda^{-1} \right)  + \sqrt{4 B^2+1} \left(\lambda - \lambda^{-1} \right)
\end{split}
\end{equation}
Note that for all $\lambda \in \mathbb{C}^\ast$ we have
\begin{equation} \label{eq:mu-symmetry}
	\mu_-(\lambda^{-1}) = \mu_+(\lambda)\,.
\end{equation}
Now $\mu_-$ has simple roots at $\varrho,\,r$, while $\mu_+$ has simple roots at $\varrho^{-1},\,r^{-1}$ , with  $\varrho,\,r,\,\varrho^{-1},\,r^{-1} $ as in \eqref{eq:roots}. While the eigenvalues of a K-matrix depend on $\lambda$ as well as both constants $A$ and $B$, the eigenvectors do not depend on $\lambda$, nor the constant $A$.
\begin{proposition}
The eigenvectors of $K$ are given by
\begin{equation} \label{eq:eigenvectors}
	\mathbf{v} = \begin{pmatrix} - \sqrt{4 B^2+1}-2 B \\ 1 \end{pmatrix} \quad \mbox{ and } \quad \mathbf{v}^\perp = \begin{pmatrix} \sqrt{4 B^2+1}-2 B \\ 1 \end{pmatrix} \,.
\end{equation}
so for all $\lambda \in \mathbf{C}^\ast$ we have 
\begin{equation*}
	K(\lambda) \,\mathbf{v} = \mu_-(\lambda) \,\mathbf{v} \quad \mbox{ and } \quad K(\lambda) \,\mathbf{v}^\perp = \mu_+ (\lambda) \,\mathbf{v}^\perp \,.
\end{equation*}
\end{proposition}
For the orthogonal matrix $V$ whose columns are $\mathbf{v},\mathbf{v}^\perp$, that is
\begin{equation*}
 	V = \begin{pmatrix} - \sqrt{4 B^2+1}-2 B &  \sqrt{4 B^2+1}-2 B \\ 1 & 1 \end{pmatrix} 
\end{equation*}
we have the diagonalization of $K$ by
\begin{equation} \label{eq:Kdiagonal}
	K = V\, \mathrm{diag}[\,\mu_-,\,\mu_+\,] \,V^{-1}\,.
\end{equation}
The span $\langle \,\,\rangle$ of the eigenvectors $\mathbf{v},\,\mathbf{v}^\perp$ of $K$ coincide with the kernels in \eqref{eq:kernels}, and we have 
\begin{equation*} \begin{split}
	\langle \mathbf{v} \rangle &= \ker K (\varrho) = \ker K (r) \\
	\langle \mathbf{v}^\perp \rangle &= \ker K (\varrho^{-1}) = \ker K (r^{-1})\,.
\end{split}
\end{equation*}
Suppose we have a K-symmetric potential $\xi_\lambda$. Evaluating  the symmetry \eqref{eq:Sklyanin-potentials} at $\lambda = \varrho$ and applying to the vector $\mathbf{v} \in \ker K(\varrho)$ gives 
\begin{equation*} 
	K(\varrho)\,\xi_\varrho\,\mathbf{v} + \overline{\xi_\varrho}^t K(\varrho)\,\mathbf{v} = K(\varrho)\,\xi_\varrho\,\mathbf{v}  = 0\,,
\end{equation*}
and hence $\xi_\varrho\,\mathbf{v} \in \ker K(\varrho) = \langle \mathbf{v} \rangle$. Similarly,  $\xi_r\,\mathbf{v} \in \ker K(r) = \langle \mathbf{v} \rangle$, $\xi_{\varrho^{-1}}\,\mathbf{v}^\perp \in \ker K(\varrho^{-1}) = \langle \mathbf{v}^\perp \rangle$ and $\xi_{r^{-1}}\,\mathbf{v}^\perp \in \ker K(r^{-1})  = \langle \mathbf{v}^\perp \rangle$. We summarize this in the following 
\begin{proposition} At each of the four points $\lambda \in \{ \varrho,\,r,\,\varrho^{-1},\,r^{-1} \}$ where $K(\lambda)$ drops rank, the kernel of $K(\lambda)$ is an eigenspace of $\xi_\lambda$. We may assume without loss of generality that 
\begin{equation*} 
	\xi_\varrho \mathbf{v} = \xi_r \mathbf{v} = -\nu\,\mathbf{v} \quad \mbox{ and } \quad
	\xi_{\varrho^{-1}} \mathbf{v}^\perp = \xi_{r^{-1}} \mathbf{v}^\perp = \nu\,\mathbf{v}\,.
\end{equation*} 
\end{proposition}

%
%
%
\subsection{The Inverse and Residues.} Now $K^{-1} = V\, \mathrm{diag}[\,\mu_-^{-1},\,\mu_+^{-1}\,] \,V^{-1}$ has simple poles at the four points $\lambda \in \{ \varrho,\,r,\,\varrho^{-1},\,r^{-1} \}$ since $\mu_\pm$ have simple roots there. Since the matrix $V$ does not depend on $\lambda$, the residues of $K^{-1}$ at the four points $\lambda \in \mathcal{N}$ as in \eqref{eq:nullset} are given by 
\[
 \mathrm{res}[K^{-1},\lambda]	  = V\, \mathrm{diag}\bigl[ \,\mathrm{res}[\mu_-^{-1},\,\lambda],\,\mathrm{res}[\mu_+^{-1},\,\lambda \,\bigr] \,V^{-1}\,, \qquad \lambda \in \mathcal{N}.
\]
\begin{lemma} The kernels of the residues of $K^{-1}$ coincide with the kernels of $K$ as follows
\begin{equation*} \begin{split}
	\ker \mathrm{res}[K^{-1},\varrho] &= \ker K(\varrho^{-1}) = \ker \mathrm{res}[K^{-1},r] = \ker K(r^{-1}) \,,\\
	\ker \mathrm{res}[K^{-1},\varrho^{-1}] &= \ker K(\varrho) = \ker \mathrm{res}[K^{-1},r^{-1}] = \ker K(r)\,.
\end{split}
\end{equation*}
\end{lemma}
\begin{proof}
To compute the residues, first multiply away the poles
\begin{equation*} \begin{split}
	(\lambda-\varrho) \,\mu_-^{-1} &= ( (\sqrt{4 A^2+1}+2 A)\lambda^{-1} - \sqrt{4 B^2+1}-2 B)^{-1} \\
	(\lambda-r) \,\mu_-^{-1} &= - ((\sqrt{4 A^2+1}-2 A)\lambda^{-1}+\sqrt{4 B^2+1} +2 B )^{-1} \\
	(\lambda - \varrho^{-1})\, \mu_+^{-1} &= ((-\sqrt{4 A^2+1}+2 A)\lambda^{-1}+\sqrt{4 B^2+1}  -2 B )^{-1} \\
	(\lambda- r^{-1}) \,\mu_+^{-1} &= ((\sqrt{4 A^2+1}+2 A)\lambda^{-1}+\sqrt{4 B^2+1}  -2 B )^{-1}
\end{split}
\end{equation*}
then evaluate to obtain the four residues
\begin{equation*} \begin{split}
	\mathrm{res}[\mu_-^{-1},\,\varrho] = \lim_{\lambda \to \varrho} (\lambda-\varrho)\,\mu_-^{-1} &= (\tfrac{1}{2} - \tfrac{A}{\sqrt{4A^2+1}} ) (2B - \sqrt{4 B^2+1} )\\
	\mathrm{res}[\mu_-^{-1},\,r] = \lim_{\lambda \to r}(\lambda-r)\,\mu_-^{-1} &=  (  \tfrac{1}{2} + \tfrac{A}{\sqrt{4 A^2+1}} ) (2B - \sqrt{4 B^2+1}) \\
	\mathrm{res}[\mu_+^{-1},\,\varrho^{-1}] = \lim_{\lambda \to \varrho^{-1}}(\lambda - \varrho^{-1})\,\mu_+^{-1} &=  ( \tfrac{1}{2} + \tfrac{A}{\sqrt{4 A^2+1}}) ( 2B + \sqrt{4 B^2+1}) \\
	\mathrm{res}[\mu_+^{-1},\,r^{-1}] = \lim_{\lambda \to r^{-1}}(\lambda- r^{-1})\,\mu_+^{-1} &= ( \tfrac{1}{2}-\tfrac{A}{\sqrt{4 A^2+1}}) ( 2B + \sqrt{4 B^2+1})
\end{split}
\end{equation*}
The residues of $K^{-1}$ at the four simple poles compute to
\begin{equation*} \begin{split}
	\mathrm{res}[K^{-1},\varrho] &= V \left( \begin{smallmatrix}\mathrm{res}[\mu_-^{-1},\varrho] & 0 \\ 0 & 0 \end{smallmatrix} \right) V^{-1} \\ &= 
\tfrac{\sqrt{4 A^2+1}-2A}{4 \sqrt{4 A^2+1} \sqrt{4 B^2+1}} \left( \begin{smallmatrix}
 -1 & \sqrt{4B^2+1}-2 B \\ \sqrt{4 B^2+1}-2 B & - (\sqrt{4 B^2+1}-2 B)^2\\
\end{smallmatrix} \right)  \sim \left( \begin{smallmatrix} 1 & 2 B - \sqrt{4 B^2 + 1} \\ 0 & 0  \end{smallmatrix} \right) \\
	\mathrm{res}[K^{-1},r] &= V \left( \begin{smallmatrix} \mathrm{res}[\mu_-^{-1},r] & 0 \\ 0 & 0 \end{smallmatrix} \right) V^{-1} \\ &= \tfrac{\sqrt{4 A^2+1}+2A}{4 \sqrt{4 A^2+1} \sqrt{4 B^2+1}} \left( \begin{smallmatrix}
 -1 & \sqrt{4B^2+1}-2 B \\ \sqrt{4 B^2+1}-2 B & - (\sqrt{4 B^2+1}-2 B)^2\\
\end{smallmatrix} \right) \sim \left( \begin{smallmatrix} 1 & 2 B - \sqrt{4 B^2 + 1} \\ 0 & 0  \end{smallmatrix} \right)  \\
	\mathrm{res}[K^{-1},\varrho^{-1}] &= V \left( \begin{smallmatrix} 0 & 0 \\ 0 & \mathrm{res}[\mu_+^{-1},\varrho^{-1}] \end{smallmatrix} \right) V^{-1} \\ &=  \tfrac{\sqrt{4 A^2+1}+2A}{4 \sqrt{4 A^2+1} \sqrt{4 B^2+1}} \left( \begin{smallmatrix}
 -1 & \sqrt{4B^2+1}+2 B \\ \sqrt{4 B^2+1}+2 B & - (\sqrt{4 B^2+1}+2 B)^2\\
\end{smallmatrix} \right) \sim \left( \begin{smallmatrix} 1 & 2 B + \sqrt{4 B^2 + 1} \\ 0 & 0  \end{smallmatrix} \right)  \\
	\mathrm{res}[K^{-1},r^{-1}] &= V \left( \begin{smallmatrix} 0 & 0 \\ 0 & \mathrm{res}[\mu_+^{-1},r^{-1}]\end{smallmatrix} \right) V^{-1} \\ &=  \tfrac{\sqrt{4 A^2+1}-2A}{4 \sqrt{4 A^2+1} \sqrt{4 B^2+1}} \left( \begin{smallmatrix}
 -1 & \sqrt{4B^2+1}+2 B \\ \sqrt{4 B^2+1}+2 B & - (\sqrt{4 B^2+1}+2 B)^2\\
\end{smallmatrix} \right) \sim \left( \begin{smallmatrix} 1 & 2 B + \sqrt{4 B^2 + 1} \\ 0 & 0  \end{smallmatrix} \right) 
\end{split}
\end{equation*}
Comparing row-reduced matrices in \eqref{eq:kernels} and the above row-reduced residues proves the claim.
\end{proof}
%
\section{Products of K-matrices}
%
We next look at products and ratios of K-matrices. Suppose we have two K-matrices 
\begin{equation} \label{eq:K0K1} 
	K_0  = \begin{pmatrix}
 4 A_0  -4 B_0 \lambda  & \lambda - \lambda^{-1} \\
 \lambda -  \lambda^{-1} & 4 A_0  - 4 B_0 \lambda^{-1}
\end{pmatrix},\, K_1  = \begin{pmatrix}
 4 A_1  -4 B_1 \lambda  & \lambda - \lambda^{-1} \\
 \lambda -  \lambda^{-1} & 4 A_1  - 4 B_1 \lambda^{-1}
\end{pmatrix}.
\end{equation}
Then 
\[
	[K_0,\, K_1] = 4  \left( \lambda - \lambda^{-1} \right)^2 (B_0 - B_1)
\begin{pmatrix}
 0 & -1 \\ 1 & 0 
\end{pmatrix}
\]
and thus
\begin{equation} \label{eq:Commutator0}
	[K_0,\, K_1] = 0 \quad \mbox{for all  } \lambda \in \mathbb{C}^\ast \Longleftrightarrow B_0 = B_1\,.
\end{equation}
\begin{proposition} \label{th:K1K0}
Suppose $K_0,\,K_1$ are two K-matrices with 
\begin{equation} \label{eq:sign-change}
	A_0 B_1 = - A_1 B_0\,.
\end{equation}
Then there exist unique rational functions $p,\,q$, and a degree one  potential $\eta_\lambda$ with
\begin{equation} \label{eq:K1K0}
	K_1^{-1} K_0 = p\,\mathbbm{1} +  q\,\eta_\lambda
\end{equation}
\end{proposition}
\begin{proof}
We compute
\begin{equation*} \begin{split}
	\mathrm{adj} K_1 K_0 &= \begin{pmatrix}
 4 A_1  -4 B_1 \lambda^{-1}  & -(\lambda - \lambda^{-1}) \\
 -(\lambda -  \lambda^{-1}) & 4 A_1  - 4 B_1 \lambda
\end{pmatrix}   \begin{pmatrix}
 4 A_0  -4 B_0 \lambda  & \lambda - \lambda^{-1} \\
 \lambda -  \lambda^{-1} & 4 A_0  - 4 B_0 \lambda^{-1}
\end{pmatrix} \\
	= &\bigl( 16(A_0A_1+B_0B_1) - (\lambda-\lambda^{-1})^2 \bigr) \,\mathbbm{1} \,\, + \\
	\lambda^{-1}&\begin{pmatrix} -16(A_0B_1 + A_1B_0 \lambda^2) & 4(A_1-A_0-(B_1-B_0)\lambda^{-1})(\lambda^2-1) \\
4(A_1-A_0-(B_1-B_0)\lambda)(\lambda^2-1) & -16(A_1 B_0 + A_0 B_1\lambda^2) \end{pmatrix}
\end{split}
\end{equation*}
The last matrix is trace-free if and only if \eqref{eq:sign-change} holds, and then computes to
\begin{equation*} \begin{split}
	\begin{pmatrix} -16(A_0B_1 +A_1  B_0 \lambda^2) & 4(A_1-A_0-(B_1-B_0)\lambda^{-1})(\lambda^2-1) \\
4(A_1-A_0-(B_1-B_0)\lambda)(\lambda^2-1) & -16(A_1B_0 + A_0 B_1\lambda^2) \end{pmatrix}& \\ = 4 (\lambda^2-1) \begin{pmatrix} 4A_0 B_1 & A_1-A_0-(B_1-B_0)\lambda^{-1} \\
A_1-A_0-(B_1-B_0)\lambda & -4A_0 B_1 \end{pmatrix}&\,.
\end{split}
\end{equation*}
Replacing $\mathrm{adj}K_1 = \det K_1 K_1^{-1}$ and putting the above computation together we set 
\begin{equation} \label{eq:fgK1K0} \begin{split}
p(\lambda) &=\bigl( 16(A_0A_1+B_0B_1) - (\lambda-\lambda^{-1})^2 \bigr)/ \det K_1 (\lambda)\,, \\
q(\lambda) &= - 4\mi (\lambda - \lambda^{-1})/ \det K_1 (\lambda)\,, \\
\eta_\lambda &= \mi  \begin{pmatrix} 4A_0B_1 & A_1-A_0-(B_1-B_0)\lambda^{-1} \\
A_1-A_0-(B_1-B_0)\lambda & -4A_0B_1 \end{pmatrix}\,.
\end{split}
\end{equation}
\end{proof}
We next show that potentials that are K-symmetric with respect to two distinct K-matrices only give rise to Delaunay surfaces. 
\begin{corollary}
Let $K_0,\,K_1$ be two K-matrices satisfying \eqref{eq:sign-change}, and $\xi_\lambda$ a potential that is K-symmetric with respect to both $K_0$ and $K_1$. Then $\xi_\lambda$ is a polynomial multiple of a degree 1 potential.
\end{corollary}
\begin{proof}
Suppose we have two different K-matrices $K_0$ and $K_1$ that satisfy the condition \eqref{eq:sign-change}. Suppose we have a potential $\xi_\lambda$ that is K-symmetric with respect to both $K_0$ and $K_1$. Then
\[
	\overline{\xi_{\bar\lambda}}^t = -K_0\xi_\lambda K_0^{-1} = -K_1\xi_\lambda K_1^{-1} 
\]	
holds if and only if
\begin{equation}
	[ K_1^{-1} K_0,\,\xi_\lambda ] = 0\,.
\end{equation}
Using \eqref{eq:K1K0} gives $[  p\,\mathbbm{1} +  q\,\eta_\lambda\,,\,\xi_\lambda ] = 0$ and hence $[\eta_\lambda,\,\xi_\lambda ] =0$. Thus $\xi_\lambda = r_\lambda \eta_\lambda$ for some polynomial $r_\lambda$. 
\end{proof}
\subsection{The case $B_0 = B_1$.} If we diagonalize as in \eqref{eq:Kdiagonal}, then $K_0 = V_0\, \mathrm{diag}[\,\mu^0_-,\,\mu^0_+\,] \,V_0^{-1}$ and $K_1^{-1} = V_1\, \mathrm{diag}[\,1/\mu^1_-,\,1/\mu^1_+\,] \,V_1^{-1}$ . If the constant $B$ in the K-matrices $K_0$ and $K_1$ coincide $B_0 = B_1$, then the matrix of eigenvectors $V=V_0 = V_1$ are equal, and 
\begin{equation} 
	K_1^{-1} K_0 = V \,\mathrm{diag}[\,\mu^0_-/\mu^1_-,\,\mu^0_+/\mu^1_+\,]  V^{-1} \,.
\end{equation}
%

%
%
%
\section{Two integrable boundary conditions} 

Suppose in our domain $U$ we have beside the curve $y=0$ a second curve $y=y_1 \neq 0$, and a  solution $\omega$ of the sinh-Gordon equation that satisfies two boundary conditions
\begin{equation} \label{eq:2-boundary1}
	\begin{split}
	&\omega_y = e^\omega A_0 + e^{-\omega} B_0 \quad \mbox{ along } y=0 \,,\\
	& \omega_y = e^\omega A_1 + e^{-\omega} B_1 \quad \mbox{ along } y=y_1\,.
\end{split}
\end{equation}
We then have two K-matrices $K_0,\,K_1$ as in \eqref{eq:K0K1} and
\begin{equation}\begin{split}
	K_0 \mathbf{U}_\lambda &= \mathbf{U}_{\lambda^{-1}} K_0 \qquad \mbox{along} \quad y=0  \\
	K_1 \mathbf{U}_\lambda &= \mathbf{U}_{\lambda^{-1}} K_1 \qquad \mbox{along} \quad y=y_1\,.
\end{split}
\end{equation}
In the 1-boundary case the $y=0$ curve goes through the base point $z=0$, and Lemma \ref{lem:Fswitch2} assures us $K\,\mathbf{F}_\lambda =  \mathbf{F}_{\lambda^{-1}}\,K$ along $y=0$. The $y=y_1$ curve does not contain the base point at which $\mathbf{F}_\lambda (0) = \mathbbm{1}$, and we only get a \emph{dressed} $K$-symmetry.
\begin{lemma} \label{lem:Fswitch}
$\mathrm{(1)}$ Suppose $\tfrac{\partial}{\partial x}\mathbf{F}_\lambda = \mathbf{F}_\lambda \mathbf{U}_\lambda$. Then $K_1\,\mathbf{U}_\lambda = \mathbf{U}_{\lambda^{-1}} \,K_1$ along $y=y_1$, if and only if  there exists a $z$-independent matrix $C_\lambda$ such that
\begin{equation} \label{eq:FK1}
	 K_1\,\mathbf{F}_\lambda =  C_\lambda \,\mathbf{F}_{\lambda^{-1}}\,K _1 \qquad \mbox{\rm{along}} \quad y=y_1\,.
\end{equation}
$\mathrm{(2)}$ The map $\lambda \mapsto C_\lambda$ is holomorphic on $\mathbb{C}^\ast$. Further, $\det C_\lambda = 1$ for all $\lambda \in \mathbb{C}^\ast$,  and $C_\lambda \in \mathrm{SU}_2$ for all $|\lambda |=1$.
\end{lemma}
\begin{proof}
(1) Suppose $K_1\,\mathbf{U}_\lambda = \mathbf{U}_{\lambda^{-1}} \,K_1$ along $y=y_1$. Consider $G_\lambda = K_1\,F_\lambda\,K_1^{-1}$. Then $G'_\lambda = G_\lambda \mathbf{U}_{\lambda^{-1}}$ and $F'_{\lambda^{-1}} =   F_{\lambda^{-1}} \mathbf{U}_{\lambda^{-1}}$. Since $G_\lambda$ and $F_{\lambda^{-1}}$ are both fundamental solutions of the same ordinary differential equation, there exists a $z$-independent, $\lambda$-dependent matrix $C_\lambda$ such that $G_\lambda = C_\lambda F_{\lambda^{-1}}$.  Thus $ K_1\,F_\lambda\,K_1^{-1} = C_\lambda F_{\lambda^{-1}}$ and the claim follows. Note that 
\begin{equation}\label{eq:defC}
	C_\lambda = K_1\,F_\lambda\,K_1^{-1} F_{\lambda^{-1}}^{-1} = K_1\,F_\lambda\,K_1^{-1} \overline{F_{\bar\lambda}}^{t} \,. 
\end{equation}
Conversely, differentiating \eqref{eq:FK1} gives $K_1\,\mathbf{F}_\lambda \mathbf{U}_\lambda=  C_\lambda \,\mathbf{F}_{\lambda^{-1}} \mathbf{U}_{\lambda^{-1}}\,K _1$ and using \eqref{eq:FK1} gives the claim.

(2) By definintion $C_\lambda = G_\lambda F_{\lambda^{-1}}^{-1}$, and since $G_\lambda, \,F_\lambda$ are holomorphic on $\mathbb{C}^\ast$, the same is true for $C_\lambda$. Similarly, $C_\lambda$ inherits its other asserted properties. 
\end{proof}
\begin{corollary}
Suppose $\mathbf{F}_\lambda$ satisfies \eqref{eq:FK1} and $\zeta_\lambda = \mathbf{F}_\lambda^{-1} \xi_\lambda \mathbf{F}_\lambda$ is a polynomial Killing field. Then along $y=y_1$ we have
\begin{equation} \label{eq:K1pKf}
	K_1\,\zeta_\lambda + \overline{\zeta_{\bar{\lambda}}}^t K_1  = \mathbf{F}_{\lambda^{-1}}^{-1} \bigl( \, C_\lambda^{-1} K_1 \,\xi_\lambda + \overline{\xi_{\bar{\lambda}}}^t C_\lambda^{-1} K_1 \,\bigr) \,\mathbf{F}_\lambda\,.
\end{equation}
\end{corollary}
Hence $\zeta_\lambda$ is $K_1$-symmetric along $y=y_1$ if and only if its potential $\xi_\lambda$ satisfies the dressed $K_1$-symmetry 
\begin{equation}\label{eq:dressedK-symmetry}
	C_\lambda^{-1} K_1 \xi_\lambda + \overline{\xi_{\bar\lambda}}^t C_\lambda^{-1} K_1 = 0\,.
\end{equation}
If in addition the potential is $K_0$-symmetric, then we can replace $\overline{\xi_{\bar{\lambda}}}^t = - K_0 \xi_\lambda K_0^{-1}$ in \eqref{eq:K1pKf} and obtain along $y=y_1$ that
\begin{equation} \label{eq:K1pKf2}
	K_1\,\zeta_\lambda + \overline{\zeta_{\bar{\lambda}}}^t K_1  = \mathbf{F}_{\lambda^{-1}}^{-1} \, K_0 \,\, \bigl[ \, K_0^{-1}  C_\lambda^{-1} K_1 ,\,\xi_\lambda \,\bigr] \,\mathbf{F}_\lambda\,.
\end{equation}
Then $\zeta_\lambda$ is $K_1$-symmetric along $y=y_1$ if and only if 
\begin{equation} \label{eq:M}
	[M_\lambda,\,\xi_\lambda ] = 0 \quad \mbox{ for } \quad M_\lambda =  K_1^{-1} C_\lambda K_0 \,.
\end{equation}
\begin{corollary}
Suppose $\mathbf{F}_\lambda$ satisfies \eqref{eq:FK1} and $\zeta_\lambda = \mathbf{F}_\lambda^{-1} \xi_\lambda \mathbf{F}_\lambda$ is a polynomial Killing field with $K_0$-symmetric potential. Then $\zeta_\lambda$ is $K_1$-symmetric along $y=y_1$ if and only if \eqref{eq:M} holds.
\end{corollary}
The proof of the next result is analogous to the proof of Lemma~\ref{th:KPhi}.
\begin{lemma} 
Let $\xi_\lambda$ satisfy the dressed $K_1$-symmetry \eqref{eq:dressedK-symmetry}. Then  
\begin{equation} \label{eq:dressedPhiK-symmetry}
	K_1\,\mathbf{\Phi}_\lambda (z) = C_\lambda \,\overline{\mathbf{\Phi}_{\bar\lambda} (\bar z)} ^{t^{-1}} C_\lambda^{-1} K_1 \qquad \mbox{for all $z \in U$}\,.
\end{equation}
\end{lemma}
Finally, we compute the dressed $K_1$-summetry for $\mathbf{B}_\lambda$ from \eqref{eq:FK1} and \eqref{eq:dressedPhiK-symmetry}.
\begin{corollary}
Let $\xi_\lambda$ satisfy the dressed $K_1$-symmetry \eqref{eq:dressedK-symmetry} and the corresponding $\mathbf{F}_\lambda$ satisfy the $K_1$-symmetry \eqref{eq:FK1}along $y=y_1$. Then the dressed $K_1$-symmetry of $\mathbf{B}_\lambda$ is
\begin{equation} \label{eq:dressedBK-symmetry}
	K_1 \mathbf{B}_\lambda = \overline{\mathbf{B}}_{\bar\lambda}^{t^{-1}} C_\lambda^{-1} K_1 \qquad \mbox{ along \quad $y=y_1$}\,.
\end{equation}
\end{corollary}
%
In summary, in the 2-boundary case we have a polynomial Killing field $\zeta_\lambda = \mathbf{F}_\lambda^{-1} \xi_\lambda \mathbf{F}_\lambda$ where $\mathbf{F}_\lambda$ satisfies \eqref{eq:KF} and \eqref{eq:FK1}, and  
\begin{equation} \label{eq:2-boundary}
	\begin{split}
	K_0\,\zeta_\lambda + \overline{\zeta_{\bar{\lambda}}}^t K_0 &= 0 \quad \mbox{ along } y=0  \quad \Longleftrightarrow \quad K_0\,\xi_\lambda + \overline{\xi_{\bar{\lambda}}}^t K_0 = 0 \,,\\
	K_1\,\zeta_\lambda + \overline{\zeta_{\bar{\lambda}}}^t K_1 &= 0 \quad \mbox{ along } y=y_1  \quad \Longleftrightarrow \quad C_\lambda^{-1} K_1 \xi_\lambda + \overline{\xi_{\bar\lambda}}^t C_\lambda^{-1} K_1 = 0\,.
\end{split}
\end{equation}
%
%
\section{The commutator} 
\begin{proposition}
The map $\lambda \mapsto M_\lambda$ in \eqref{eq:M} is meromorphic on $\mathbb{C}^\ast$ with exactly four simple poles at the four simple roots of $\det K_1$.  Also $\det M_\lambda = \det K_0 /\det K_1$.
\end{proposition}
\begin{proof}
The four simple poles are the four simple roots of $\det K_1$. The last assertion follows from $\det C_\lambda = 1$ for all $\lambda \in \mathbb{C}^\ast$.
\end{proof}
Since $[M_\lambda,\,\xi_\lambda] = 0$, there exist $\lambda$-dependent functions $f,\,g$ such that 
\begin{equation}\label{eq:f-g}
	M_\lambda  = f\,\mathbbm{1} + g\, \xi_\lambda\,.
\end{equation}
Taking determinants gives
\begin{equation*}
	 \frac{\det K_0}{\det K_1} = f^2 + g^2 \det \xi_\lambda\,.
\end{equation*}
In terms of the eigenvalues, $\det K_0 = \mu^0_- \mu^0_+,\,\det K_1 =  \mu^1_- \mu^1_+$, where $\mu_\pm^0,\,\mu_\pm^1$ are as in \eqref{eq:mu_pm} with $A,\,B$ replaced by $A_0,\,B_0$ respectively $A_1,\,B_1$, and writing $\nu^2 = -\det \xi_\lambda$ we obtain
\begin{equation*}
	 f^2 - g^2 \nu^2 = (f-g\,\nu)(f + g\,\nu) = \mu^0_- \mu^0_+/(\mu^1_- \mu^1_+)
\end{equation*}
Setting	$f-g\,\nu =  \mu^0_-/\mu^1_-$ and $f +g\,\nu =  \mu^0_+/ \mu^1_+$ gives
\begin{equation*} 
	f = \frac{1}{2} \left(  \frac{\mu^0_-}{ \mu^1_-} +  \frac{\mu^0_+}{ \mu^1_+} \right) \qquad \mbox{ and } \qquad
	g = \frac{1}{2\,\nu} \left(    \frac{\mu^0_+}{ \mu^1_+} - \frac{\mu^0_-}{ \mu^1_-} \right) \,.
\end{equation*}
By \eqref{eq:mu-symmetry}, and since $\nu(\lambda^{-1}) = \lambda^{1-d} \nu ( \lambda) $ we have
\begin{equation*}
	f(\lambda^{-1}) = f(\lambda) \quad \mbox{ and } \quad g(\lambda^{-1}) = -  \lambda^{d-1} g(\lambda) \,.
\end{equation*}
%
%
\subsection{Kernels.} 
We briefly verify that $f\mathbbm{1} + g\,\xi_\lambda$ has the given kernels at the four roots of $\det K_0$. Suppose the four roots of $\det K_0$ are $\{\varrho_0,\,r_0,\,\varrho_0^{-1},\,r_0^{-1} \}$ as in \eqref{eq:kernels} with $A,\,B$ replaced by $A_0,\,B_0$. Suppose further that $\mu^0_- (\varrho_0) = \mu^0_- (r_0) = 0$, and that $\langle \mathbf{v}_0 \rangle = \ker K_0(\varrho_0) = \ker K_0 (r_0)$. Then $M_{\varrho_0} \mathbf{v}_0 = M_{r_0} \mathbf{v}_0 = 0$ and as we may assume that $\xi_{\varrho_0} \mathbf{v}_0 = \xi_{r_0} \mathbf{v}_0 =-\nu\,\mathbf{v}_0$ we obtain
\begin{equation*} \begin{split}
	\left( f(\varrho_0)\,\mathbbm{1} + g(\varrho_0) \,\xi_{\varrho_0} \right) \mathbf{v}_0 &= \frac{1}{2} \frac{\mu^0_+(\varrho_0)}{ \mu^1_+(\varrho_0)}  \left( \mathbbm{1} + \frac{1}{\nu}\xi_{\varrho_0} \right)  \mathbf{v}_0  = 0\,, \\
\left( f(r_0)\,\mathbbm{1} + g(r_0) \,\xi_{r_0} \right) \mathbf{v}_0 &= \frac{1}{2} \frac{\mu^0_+(r_0)}{ \mu^1_+(r_0)} \left( \mathbbm{1} + \frac{1}{\nu}\xi_{r_0} \right)  \mathbf{v}_0 = 0\,.
\end{split}
\end{equation*} 
Similarly, if $\mu^0_+ (\varrho_0^{-1}) = \mu^0_+ (r_0^{-1}) = 0$, and that $\langle \mathbf{v}_0^\perp  \rangle = \ker K_0(\varrho_0^{-1}) = \ker K_0 (r_0^{-1})$, then $M_{\varrho_0^{-1}} \mathbf{v}_0^\perp = M_{r_0^{-1}} \mathbf{v}_0^\perp = 0$ and as we may assume that $\xi_{\varrho_0^{-1}} \mathbf{v}_0^\perp = \xi_{r_0^{-1}} \mathbf{v}_0^\perp =\nu\,\mathbf{v}_0^\perp$, we obtain
\begin{equation*} \begin{split}
	\left( f(\varrho_0^{-1})\,\mathbbm{1} + g(\varrho_0^{-1}) \,\xi_{\varrho_0^{-1}} \right) \mathbf{v}_0^\perp &= \frac{1}{2}  \frac{\mu^0_-(\varrho_0^{-1})}{ \mu^1_-(\varrho_0^{-1})}  \left( \mathbbm{1} - \frac{1}{\nu}\xi_{\varrho_0^{-1}} \right)  \mathbf{v}_0^\perp =  0\,, \\
\left( f(r_0^{-1})\,\mathbbm{1} + g(r_0^{-1}) \,\xi_{r_0^{-1}} \right) \mathbf{v}_0^\perp &= \frac{1}{2} \frac{\mu^0_-(r_0^{-1})}{ \mu^1_-(r_0^{-1})} \left( \mathbbm{1} - \frac{1}{\nu}\xi_{r_0^{-1}} \right)  \mathbf{v}_0^\perp = 0\,.
\end{split}
\end{equation*} 
%
%
\subsection{Spectrum.} 
The eigenvalues of $f\mathbbm{1} + g\,\xi_\lambda$ are $f \pm g\nu$. From above we have $f-g\,\nu =  \mu^0_-/ \mu^1_-$ and $f+ g\,\nu =  \mu^0_+/\mu^1_+$ so both eigenvalues of $M_\lambda$  are rational with two simple roots, and two simple poles, and principal divisors
\begin{equation*}
	(f - g\,\nu ) = \varrho_0 + r_0 - \varrho_1 - r_1 \quad \mbox { and } \quad (f + g\,\nu) = \varrho_0^{-1} + r_0^{-1} -   	\varrho_1^{-1} - r_1^{-1}\,.
\end{equation*}
Since $[ M_\lambda,\,\xi_\lambda ] =0$, these two matrices have the same eigenvectors, and thus
\begin{equation*}
	\xi_\lambda \mathbf{v} = \pm \nu \,\mathbf{v} \Longleftrightarrow  M_\lambda \mathbf{v} = (f  \pm g\,\nu)\, \mathbf{v} \,.
\end{equation*}
%
%

%
%
\subsection{Complementary boundary conditions} 

Of particular interest is the case when the two integrable boundary conditions are \emph{complimentary}, that is when the constants in the two K-matrices $K_0,\,K_1$ have opposite signs so that 
\begin{equation*} 
	A_1=-A_0 \quad \mbox{ and } B_1 = -B_0.
\end{equation*}
Omitting subscripts, and reverting back to constants $A,\,B$ we thus consider 
\begin{equation*} 
	K_0  = \begin{pmatrix}
 4 A  -4 B \lambda  & \lambda - \lambda^{-1} \\
 \lambda -  \lambda^{-1} & 4 A  - 4 B \lambda^{-1}
\end{pmatrix} \, \mbox{ and } \, 
	K_1  = \begin{pmatrix}
 -4 A  +4 B \lambda  & \lambda - \lambda^{-1} \\
 \lambda -  \lambda^{-1} & -4 A  + 4 B \lambda^{-1}
\end{pmatrix}\,.
\end{equation*}
From \eqref{eq:mu_pm} we know that $K_0$ has eigenvalues
\begin{equation*}  \begin{split}
	\mu_- ^0(\lambda) &= 4 A  - 2 B \left(\lambda + \lambda^{-1} \right)  - \sqrt{4 B^2+1} \left(\lambda - \lambda^{-1} \right) \\
\mu_+^0  (\lambda) &= 4 A  - 2 B \left(\lambda + \lambda^{-1} \right)  + \sqrt{4 B^2+1} \left(\lambda - \lambda^{-1} \right)
\end{split}
\end{equation*}
so under $(A,\,B) \mapsto (-A,\,-B)$ we see that $K_1$ has eigenvalues
\begin{equation} \label{eq:K1EV}  \begin{split}
	\mu_- ^1 &=  -\mu_+^0\\
\mu_+^1  &= -\mu_-^0
\end{split}
\end{equation}
Thus $\det K_0 = \det K_1$, and hence $\det M_\lambda  = f^2 - g^2\nu^2 = (f-g\,\nu)(f+g\,\nu) = 1$.  As in the previous section, and using \eqref{eq:K1EV} we have
\begin{equation*} 
	f-g\,\nu =  \frac{\mu^0_-}{ \mu^1_-} = - \frac{\mu^0_-}{ \mu^0_+}  \quad \mbox{ and } \quad f+g\,\nu =  \frac{\mu^0_+}{ \mu^1_+} = - \frac{\mu^0_+}{ \mu^0_-} 
\end{equation*}
and so indeed $(f-g\,\nu)(f+g\,\nu) = 1$. Further, solving for $f,\,g$ gives
\begin{equation*}
	f = -\frac{1}{2} \left(  \frac{\mu^0_-}{ \mu^0_+} +  \frac{\mu^0_+}{ \mu^0_-} \right) \quad \mbox{ and } \quad
	g = \frac{1}{2\,\nu} \left(    \frac{\mu^0_-}{ \mu^0_+} - \frac{\mu^0_+}{ \mu^0_-} \right) \,.
\end{equation*}
%
%
\subsection{The dressing matrix}
From \eqref{eq:M} and \eqref{eq:f-g} we have $M_\lambda =  K_1^{-1} C_\lambda K_0 = f\,\mathbbm{1} + g \,\xi_\lambda$ and  using the $K_0$-symmetry of $\xi_\lambda$ proves the next 
\begin{proposition} The dressing matrix $C_\lambda$ is given by 
\begin{equation}
	C_\lambda = K_1  K_0^{-1} \bigl( f\,\mathbbm{1} - g\,\, \overline{\xi_{\bar\lambda}}^t \bigr) \,.
\end{equation}
\end{proposition}

\bibliographystyle{amsplain}

\bibliography{ref}
\end{document}